\documentclass[11pt,leqno]{amsart}
\usepackage{amsfonts, amsmath, amsthm, amssymb,xspace}
\newtheorem{theorem}{Theorem}[section]
\newtheorem{lemma}{Lemma}[section]
\newtheorem{prop}{Proposition}[section]
\newtheorem{cor}{Corollary}[section]

\theoremstyle{definition}

\newcommand{\bggo}{\mathcal O}
\newcommand{\mf}[1]{\displaystyle{\mathfrak{#1}}}

\DeclareMathOperator{\spec}{\ensuremath{Spec}}
\DeclareMathOperator{\Gr}{\ensuremath{gr}}

\DeclareMathOperator{\ad}{\ensuremath{ad}}

\DeclareMathOperator{\Sym}{\ensuremath{Sym}}

\long\def\symbolfootnote[#1]#2{\begingroup\def\thefootnote{\fnsymbol{footnote}}
\footnote[#1]{#2}\endgroup}

\begin{document}

\title{Infinitesimal Hecke algebras of $\mathfrak{sl}_2$ in positive characteristic}
\author{ Akaki Tikaradze    }

\address{The University of Toledo \hfill\newline Department of Mathematics
\hfill\newline Toledo, Ohio, USA \hfill\newline
e-mail: {\tt tikar@math.uchicago.edu}}

\thanks{We would like to thank V. Ginzburg for suggesting the question, and A. Premet, I. Gordon for useful remarks.}
\begin{abstract}
We describe centers of infinitesimal Hecke algebra of $\mf{sl_2}$ in positive characteristic. In particular,
we show that these algebras are finitely generated  modules
over their centers, and the Azumaya and smooth loci of the centers coincide.
 
\end{abstract}

\maketitle
\section{Introduction}


In this paper, we  will work over an algebraically closed field $k$ of characteristic $p>2$. In this setting, we will  consider infinitesimal  Hecke algebras associated to the Lie algebra
$\mf{sl}_2$ and its natural representation (see [EGG] for the definition of infinitesimal Hecke algebras associated to an arbitrary reductive Lie algebra and its representation). We will be concerned with their centers and  
finite dimensional representations. Let us recall the precise definition of these algebra.

   Let $V$ be the natural 2-dimensional representation of $\mf{g}=\mf{sl}_2.$ Let us fix
   its basis elements $x, y$ so that $ ex=0, fx=y, hx=x, hy=-y,$ here $e, f, h$ denote the standard basis elements of $\mathfrak{sl}_2$. Then for any $z\in k[\Delta],$ where $\Delta$ denotes the rescaled Casimir element $h^2+4ef-2h$, one defines an algebra $H_z$ as the quotient of $\mf{U}\mf{g}\ltimes TV$ by the two sided ideal generated by the element $[x, y]-z,$ where $TV$ denotes the tensor algebra of $V$. 
Thus we get a family of algebras parametrized by elements in $k[\Delta]$. The algebra $H_z$  can be equipped with a natural algebra filtration $F^nH_z$ such that $\mf{U}\mf{g}= F^0H_z, \mf{U}\mf{g}V= F^1H_z$ and $F^nH_z=(F^1H_z)^n$. The main property of $H_z$ is that it satisfies the PBW property, namely the  natural surjection $H_0=\mf{U}(\mf{g}\ltimes V)\to \Gr(H_z)$ is an isomorphism ([K],[EGG]). Thus one can think of algebras $H_z$ as some kind of
deformations of the enveloping algebra $H=H_0$ depending on a deformation parameter $z\in k[\Delta]$. 

When the ground field $k$ has characteristic 0, we proved in [T] that the center of $H_z$ is isomorphic to the polynomial ring in one variable and
 $\Gr(\mf{Z}(H_z))=\mf{Z}(\Gr(H_z))=\mf{Z}(H)$ (where the letter $\mf{Z}$ stands for the center of an algebra). In this paper we show that if deg($z)<p-1,$ then algebra $H_z$ is a finitely generated module over its center
and just like in the characteristic 0 case $\Gr(\mf{Z}(H_z))=\mf{Z}(\Gr(H_z))=\mf{Z}(H).$ 
In this way we have a similar picture to the case of symplectic reflection algebras [BFG].  Geometrically $\spec \mf{Z}(H_z)$ turns out to be a $p$-fold ramified cover of the affine space $\mathbb{A}^5.$ At the end of the paper, we discuss irreducible modules of $H_z.$ In particular, we prove that the Azumaya and the smooth loci of $\mf{Z}(H_z)$ coincide. Analogous statements are known to be true for Cherednik algebras [BC], enveloping algebras of semi-simple
Lie algebras and quantized enveloping algebras [BG]. In the special case of $z=0,$ we describe all irreducible modules of $H$.

\section{Center and irreducible representations}

 Throughout this paper, we will use the term 'maximal vector' for elements of representations of $\mf{g}$ annihilated by $e$. Let us recall the following computation from [T] which we are going  to use later. For any $\omega \in k[\Delta]$ we have   
      $$[\omega, x]= (F(\omega)h+G(\omega))x+2eF(\omega)y.$$ 
Where $F, G:k[\Delta]\to k[\Delta]$ are linear endomorphisms of $k[\Delta]$ defined recursively as follows:
\begin{eqnarray*}
F(\Delta^{n+1})&=& 2\Delta^n+(\Delta-1)F(\Delta^n)-2G(\Delta^n)\\ 
G(\Delta^{n+1})&=&-3\Delta^n+(\Delta+3)G(\Delta^n)-2\Delta F(\Delta^n).
\end{eqnarray*}
 It is immediate that deg$F(\Delta^n)=$deg$G(\Delta^n)=n-1$ and leading coefficient of $F(\Delta^n)$ is $2n$, and the leading coefficient of $G(\Delta^n)$ is $-n(2n+1)$. When char($k$)=0, it was shown in [T] that center of $H_z$ is generated by the element $t_z=ex^2+hxy-fx^2-\frac{1}{2}hz-\omega_z$, where $\omega_z=-F^{-1}(z)+\frac{1}{2}z+\frac{1}{2}F^{-1}(G(z))$. We see from this formula that the element $t_z$ can be defined  in positive characteristic setting as long as char$(k)<p-1.$

We will also use the following anti-involution $j$ of $H_z$ defined as follows: $$j(x)=y, j(y)=x, j(h)=h, j(e)=-f, j(f)=-e.$$  Main result of the paper is the following

\begin{theorem}If deg $z<p-1$, then the center of $H_z$ is generated as an algebra over $k$ by $e^p, f^p, h^p-h, x_p, y_p,t_z$, where $x_p$ has top symbol with respect to the filtration equal to $x^p$, and $y_p$ has $y^p$. 
$\spec(\mf{Z}(H_z))$ is a finite (ramified) cover of $\mathbb{A}^5$=$\spec(k[e^p,f^p, h^p-h, x_p, y_p]$) of degree $p$.

\end{theorem}

   The proof will be divided into several steps. As a first step, we will show that the center of $H$ (the associated graded of $H_z$) is generated by $e^p, h^p-h, f^p, x^p, y^p, t$ (where $t=t_0=ex^2+hxy-fx^2$).  
 
 	We will argue by induction on the filtration degree of the central element. The first thing that needs to be checked is that \\$\mf{Z}(H)\cap \mf{U}\mf{g}=k[e^p, f^p, h^p-h]$. Since $\mf{Z}(\mf{U}\mf{g})=k[e^p,f^p, h^p-h, \Delta]$ [V],  any  $a\in \mf{Z}(H)\cap \mf{U}\mf{g}$ may be expressed as $a= \sum_{i<p} \alpha_i\Delta^i,$ where $ \alpha_i\in k[e^p, f^p, h^p-h]$. We have  $$0=[a,x]=\sum \alpha_i [\Delta^i,x],$$ thus $\sum \alpha_iF(\Delta^i)=0$ (recall that $F$ is the linear endomorphism of $k[\Delta]$ from the beginning.) Since $F(\Delta^i)$ has a degree $i-1$  (with leading coefficient $2i$), we get that  $\alpha_i=0$ for $i>0$ since elements $1,\Delta, ..., \Delta^{p-1}$ are linearly independent over $k[h^p-h,e^p,f^p]$, therefore $a\in k[h^p-h, f^p, e^p]$.

Now, we will deal with central elements of $H$ of positive degree in $x, y$. Let us begin by  making some preliminary remarks about homogeneous elements of $H$ that commute with $e$ and $y$.
Let $a$ be such an element. We may write $$a=(\sum_{i=0}^n b_iy^ix^{n-i})x^m, b_i\in \mf{U}\mf{g},b_n\neq 0.$$ From $[e,a]=0$ we get that $ib_i=-[e,b_{i-1}]$. In particular, $[e,b_n]=0$. Also from 
 $[a,y]=0$ we get that $[b_n,y]$ contains no elements from $\mf{U}\mf{g}y$ in its PBW monomial expression. Now we have the following 
\begin{lemma} Let $\alpha\in \mf{U}\mf{g}$ be an element which commutes with $e$, such that $[\alpha, y]$ has no terms from $\mf{U}\mf{g}y$ in its PBW monomial expansion. Then $\alpha$ belongs to $k[h^p-h, f^p, e]$
\end{lemma}
\begin{proof}
We will argue by induction on the filtration degree (with respect to the standard filtration of $\mf{U}\mf{g}$). Let  $\alpha\in \mf{U}\mf{g}$ be an element which satisfies the conditions of the lemma. Let us write $$\alpha=\sum f^i\alpha_i, \alpha_i\in k[h,e]$$ exactly in this order, meaning that $h$ is on the left, $e$ is on the right.
Then $[\alpha, y]$ has no $y$ if and only if $[\alpha_i,y]$ has no $y$. Let $$\beta=\sum a_{ij}h^ie^j, a_{ij}\in k$$ be such that $[\beta,y]$ has no $y$. Then $$[\beta,y]=\sum j a_{ij}h^ie^{j-1}x+\sum a_{ij}[h^i,y]e^j,$$ and $[h^i,y]=(-ih^{i-1} + $ lower powers of $h)y$, therefore $[h^i,y]e^j=-ih^{i-1}e^jy +$ terms with lower powers of $h$. Hence the highest $i$ with $a_{ij}\neq 0$ is a multiple of $p$. Then considering the element
 $\beta-(h^p-h)^{i/p}a_{ij}e^j$, we see that $\beta$ may be written as a linear combination of elements of the form $(h^p-h)
^ie^j$. Therefore $$\alpha=\sum _{i=0}^m (h^p-h)^i\beta_i, \beta_i\in k[f,e],$$  where $f$ is on the left from $e$. Then $$[e,\alpha]=0=\sum (h^p-h)^i[e,\beta_i]$$ This implies that $[e,\beta_k]=0$. Indeed, looking at the terms with the highest power of $h$, it will equal to $h^{pk}$ (term with the highest power of $h$ in $[e,\beta_k]$) (where we right $h$ on the left in the PBW expression of $[e,\beta_k]$). Thus, each $\beta_i$ commutes with $e$. Let us write $$\beta_i=\sum_{j=0}^m f^jg_j, g_j\in k[e].$$ Assume without loss of generality that $m< p.$ Now $[e,\beta_i]$ will contain $hf^{m-1}g_m$, and this  forces $m=0$.

\end{proof}
Applying the anti-involution $j,$ we get an analogous result for $f$ and $x$. 

Let $a$ as above be a homogeneous central element, we claim that $m$ is a multiple of $p$. Indeed, from $0=[f,a]$, looking at highest power of $y$, it is clear that $b_ny^{n+1}mx^{m-1}=0.$ hence $p$ divides $m$, so without loss of generality we may assume that $m=0$.

  Let us write $n=pl+m$ (not to be confused with the old $m$), $0\leq m<p$. From the above lemma, we conclude that $b_n=\alpha e^k, b_0=\alpha'f^k$ for some $\alpha, \alpha'\in k[e^p,f^p, h^p-h], k<p$. We also have $(\ad e)^{m+1}f^k=0$, but $m=2k$mod$p$ (since $a$ has a weight 0 with respect to $h$), therefore either $m=2k$, or $m=2k-p$. The latter case is impossible since then $m<k$, which may not happen by the Lemma below. Thus $m=2k$. Now consider the element $a-t^m\alpha(x^p)^l$. this is a central element divisible by $y$. Hence, proceeding by induction, we are done.

The following is well-known ([St]), but we include the proof for the sake of completeness.
\begin{lemma}
If $k<p$, then $(\ad e)^k(f^k)\neq 0.$

\end{lemma}
\begin{proof}
 If we write $(\ad e)^k(f^k)$ as a sum of monomials in the PBW basis $h,e,f$ then each term will have a degree at most $k$, and it will have a term $k!h^k$ in it. Therefore it cannot be 0.

\end{proof}
For the sake of brevity, let us denote the ring $k[e^p, h^p-h, f^p]$ by $R$. So far we have proved that $\mf{Z}(H)=R[x^p, y^p, t]$. Now we turn to the case of a nonzero parameter $z$. Using the well-known identity $(\ad a)^p=\ad (a^p$), we see that elements $e^p, h^p-h, f^p, t_z$ lie in the center of $H_z$. We have a natural injection gr($\mf{Z}(H_z))\to\mf{Z}(H)$ and we have to show that it is actually an isomorphism, so it remains to demonstrate the existence of $x_p, y_p\in\mf{Z}(H_z)$ such that they map to $x^p, y^p$ respectively. By virtue of the anti-involution $j$, it will suffice to prove the existence of $x_p$. We will prove this by analyzing maximal vectors in $H$ (with respect to the adjoint action of ad($\mf{g}$)) and manipulating the anti-involution $j$. In a sense, the proof is similar in spirit to the proof of the existence of $t_z$ in characteristic 0 [T].

Let us start by establishing certain facts about $\mf{U}\mf{g}$ which will be used in a crucial way later (These facts may be  very well-known but we could not find them in the literature).
\begin{lemma} Let $a\in\mf{U}\mf{g}$ be a maximal vector. Then $a$ lies in $R[\Delta, e]$
\end{lemma}
\begin{proof}
We may assume that $a$ is homogeneous with respect to ad$(h)$, say of weight $n$ and is not divisible by $e$ from the left. Let us write $a$ as a sum of PBW monomials $$a=\sum_{i=0}^{m} e^ig_i(h)f^j, g_i(h)\in k[h],g_0(h), g_m(h)\neq 0, 2(i-j)=nmod(p).$$ From $[e, a]=0$ we get that $[g_m(h), e]$ has no terms beginning with $e$ and $[e,f^{-\frac{n}{2}mod(p)}]=0$. Thus, $n=0$ and $g_m(h)\in R$. Hence, if we consider $g_m(h)(\frac{1}{4}\Delta)^m-a$ and apply induction on $m$ we will be done.

\end{proof}

\begin{prop} Let $\alpha\in \mf{Z}(\mf{U}\mf{g})$ such that it has no terms involving $e^p$. If $\alpha f^m=[e, b]$ for $0<k<p$ and $b\in \mf{U}\mf{g},$ then $\alpha=0$
\end{prop}
\begin{proof}
First, suppose $\alpha$ is  divisible by $f^p$. Then we may write $\alpha= \beta f^{p}$. If we write $b$ as sum of PBW monomials in $e^ih^jf^k,$ taking the commutator of each such monomial with $e$ of  never increases power of $f$ in it. This implies that some of monomials of $b$ with the power of $f$ less than $p$ must commute with $e,$ hence we may disregard it. Thus we may write $b=b'f^p$. Therefore without loss of generality we may assume that $\alpha$ is not divisible by $f^p$. 

We claim that in this case there exists a regular semi-simple character $\chi$ such that $\alpha$ will not vanish on $V_{\lambda}$. Let us recall the corresponding definitions [FP]. A character is called regular semi-simple if it lies in the coadjoint orbit under the action of $SL_2(k)$ of a character defined as follows :$$\chi(e^p)=\chi (f^p)=0, \chi(h^p-h)=c=\lambda ^p-\lambda, c\neq 0\in k,$$ and $$V_{\chi, \lambda}=\mf{U}\mf{g}_{\chi}\otimes _B k_{\lambda}$$ where $B$ is a subalgebra generated by $e, h$ and $k_{\lambda}$ is its one-dimensional representation on which $e$ vanishes and $h$ acts as a multiplication by $\lambda$ where $\lambda^p-\lambda=c$. Now the desired statement boils down to the following statement about polynomials in one variable. We thank M. Boyarchenko for the following quick proof.
\begin{lemma} If $\sum g_i(s(s+2))(s^p-s)^i=0$ for $g_i(s)\in k[s]$, deg$(g_i)<p$ then all $g_i$ must equal 0.
\end{lemma}
\begin{proof}

If we make the substitution $s=s-1$ and  replace $g_i(s)$ by $g_i(s-1),$ then we will have $$\sum_0^mg_i(s^2)(s^p-s)^i=0, g_m\neq 0, g_0\neq 0.$$ We have $$0=\sum s^ig_i(s^2)(s^{\frac{p-1}{2}}-1)^i.$$ Then, looking at even powers of $s$ we conclude that $$0=\sum_ig_{2j}(s^2)((s^p-s)^2)^j.$$ Since deg$(g_{2j}(s^2))<2p$, we get that all $g_{2j}=0$, a contradiction.

\end{proof}
As it is well-known [FP], these modules are irreducible and have as a basis $v_{\lambda}, fv_{\lambda}, ...,f^{p-1}v_{\lambda}.$ To complete the proof of the proposition, it will suffice to check that $f^m\in $ad$(e)(End(V_{\lambda}))$ only for finitely many $\lambda\in k$. This is very easy to see since if we have $f^m=[e, A]$, where $A\in End(V_\lambda)$, then $A=a_i:kf^iv_{\lambda}\to kf^{i+m+1}v_{\lambda}$, where $a_i\in k.$ If we write out the commutator condition we immediately see that $\lambda$ is a root of a polynomial in $F_p$ which is independent of $\lambda$, hence there are only finitely many such $\lambda$.

\end{proof}  
 Our goal will be to produce an element $A_x$ of weight 0 such that \\$A_x\in R[t_z, e, x]\cap F^{p-1}H_z$ and$[A_x,\Delta]=[x^p, \Delta]$. By the lemma below, this will complete the proof.

\begin{lemma} Let $A_x\in R[t_z, e, x]$ be an element of weight 0 lying in $F^{p-1}H_z,$ such that $[x^p-A_x, \Delta]=0.$ Then $x^p-A_x$ lies in the center of $H_z$.

\end{lemma}
\begin{proof}
Since $x^p-Ax$ commutes with $e, h$ we get that $[x^p-A_x, f]=0.$ Applying $\ad x$ to this, we get that [$x^p-A_x,y]=0$, hence $x^p-A_x$
 is a central element.

\end{proof}
   The following describes maximal vectors in $F^{p-1}H.$
\begin{prop} Let $A\in F^{p-1}H$  be a homogeneous (in $x, y)$ maximal vector. Then it can be expressed as a sum $B+C$, where $B$ is a linear combination of elements  of the form $\gamma e^j[x^i,\Delta], \gamma e^jx^i$,$ \gamma\in R[t]$, while  $C$  is a homogeneous maximal vector whose coefficient of highest power of $y$ lies in $\mf{Z}(\mf{U}\mf{g})$ and is not divisible by $e$.

\end{prop}
\begin{proof}
 First we establish that  $[\Delta, x]x^n$ can be expressed as a linear combination of elements of the form $\gamma e^j[x^i, \Delta], \gamma e^jx^i$. Indeed, 
\begin{eqnarray*}
[\Delta, x^n] & = & \sum_{k<n} x^k((2h-3)x+4ey)x^{n-k-1}\\
& = & n((2h-3)+4ey)x^{n-1}+c'x^{n} \\
& = & n[\Delta, x]x^{n-1}+c'x^{n},
\end{eqnarray*}
where $c'$ is some constant. Thus $B$ can be taken from the span of \\$\gamma e^i[x, \Delta]x^j, \gamma e^ix^j$  which is  a $k[f^p, h^p-h, \Delta, t]-k[x]$-module. Let $A$ be a homogeneous (with respect to $\ad h$ and in $x, y$) maximal vector of degree $n.$ Without loss of generality we may assume that it is not divisible by $x$ from the right, so we may write $A=\sum a_ix^iy^{n-i}$ with $a_0\neq 0$. We know that $a_0$ must be a maximal vector, thus by the above lemma we may write $a_0=\alpha e^m$, where $\alpha\in R[\Delta]$ and it is not divisible by $e$. Now, if $\frac{1}{2}n\leq m$, then considering one of the following elements $$\alpha e^{m-\frac{1}{2}n}t^{\frac{1}{2}n}-A, \alpha e^{m-\frac{1}{2}(n-1)}t^{\frac{1}{2}(n-1)} \frac{1}{4}[\Delta, x]-A,$$ it will be divisible by $x$ from the right and we may apply induction on the degree of $A$. 

So it just remains to deal with the case $2m< n.$ In this case $\alpha e^m=c (\ad e)^n(a_n)$ (here we also need the fact that $A$ can not be divisible by $y$ from the right, which was established earlier in the proof), where $c$ is some nonzero constant. Applying $(\ad f)^{2m}$ to both sides we get that $\alpha f^m=[e,b]$ for some $b\in \mf{U}\mf{g}$, but as the above proposition shows this can only happen when $m=0$, in which case we get an element of type $C$ from the statement of the proposition.

\end{proof}

 Passing to the associated graded, we see that we may replace $H$ by $H_z$ in the lemma above. Now let us look at $[X^p, \Delta].$ Clearly this is a maximal vector lying in $F^{p-1}H_z.$ So we may apply the above proposition. Using an analogous proposition for $y,$ we may write $$[x^p, \Delta]=[A_x, \Delta]+B_x+C,$$ where $A_x, B_x\in R[t_z,\Delta, e, x]$ and $C$ is a maximal vector whose coefficient in front of the highest power of $y$ is not divisible by $e$. However $[x^p, \Delta]$ is divisible by $e$ from the left and so are all the terms $[A_x,\Delta], B_x$ (since they have weight 0), forcing $C=0$. Analogously for $y,$ we have $[y^p, \Delta]=[A_y',\Delta]+B'_y,$ where  $A'_y, B'_y\in R[t_z, \Delta, f, y]$ all have weight 0 (with respect to ad($h$) as always). Our objective is to prove that $B_x=0$ and $A_x\in R[e, x, t_z].$ We will accomplish this by utilizing the involution $j$. If we apply it to the last equality we will get 
\begin{eqnarray*}
[x^p, \Delta]=-[j(y^p), j(\Delta)]=[j(A'_y), \Delta]-j(B'_y)=[a_x, \Delta]+B_x. 
\end{eqnarray*}
But on the other hand 
\begin{align*}
[x^p, \Delta] &=4e[x^p, f]=-4e(\ad x)^{p-2}(z)\\
[y^p, \Delta] &=4f[y^p,e]=4f(\ad y)^{p-2}(z).\\
\end{align*}
Therefore,
\begin{eqnarray*}
e^{-1}[x^p, \Delta] &=&-\frac{1}{(p-2)!}(\ad e)^{p-2}(f^{-1}[y^p, \Delta])\\ 
&=& -(\ad e)^{p-2}(f^{-1}[y^p, \Delta]).
\end{eqnarray*}
Comparing above formulas, we get
\begin{eqnarray*}
[j(A'_y), \Delta]-j(B'_y)=-e([ad((e)^{p-2}[f^{-1}A'_y, \Delta]+ad(e)^{p-2}f^{-1}B'_y]).
\end{eqnarray*}
Now if $B'_y=f^iy^{p-i}$, then 
\begin{eqnarray*}
(\ad e)^{p-2}(f^{i-1}y^{p-2i}) &=&(-1)^{i-1}(p-2)!e^{i-1}x^{p-2i}\\ 
&=&(-1)^{i-1}e^{i-1}x^{p-2i}, 
\end{eqnarray*}
so 
\begin{eqnarray*}
-j(f^iy^{p-2i})=(-1)^ie^ix^{p-2i}=(\ad e)^{p-2}(B'_y). 
\end{eqnarray*}
Thus if we write 
\begin{eqnarray*}
A'_y=\alpha _i \phi_i(\Delta)y^i, B'_y=\beta _i \psi _i(\Delta)y^i 
\end{eqnarray*}
where
$\alpha_i, \beta_i\in R[f, t_z]$; $\phi_i(\Delta), \psi_i(\Delta)\in \Delta k[\Delta],$ we will get that 
\begin{eqnarray*}
[\alpha_i [\phi_i(\Delta),y^i]+\beta_i\psi_i(\Delta)y^i, \Delta]=2B'_y. 
\end{eqnarray*}
The next couple of lemmas imply that from the last equality it follows that $B'_y=0$ and $\phi_i$ are constants (they do not contain $\Delta$) which will end the proof of existence of $x_p$. 
\begin{lemma}
  Let $A\in R[\Delta, e, x, t_z]\cap  F^{p-1}H_z$ be of weight 0. Then there does not exist  a maximal vector $B$ which can be expressed as a sum of elements of the form $\alpha_it_z^i[x, \Delta]x^j, \alpha_i\in R[\Delta, e],$ such that $[B,\Delta]=A$.
\end{lemma}
\begin{proof}
 By passing to the associated graded, it suffices to consider the case $z=0$. Assume that such a $B$ exists. We may assume that $A$ is homogeneous in $x, y$, since $[B, \Delta]=4e[B, f].$ By dividing by $e$ from the left we get $[B, f]=A_1$ with $A_1\in R[\Delta, e, t]$, wt($A$)=-2. Let us write $B=(\sum a_ix^iy^{n-i})x^k$, where $a_0\neq 0$. After picking terms with the lowest powers of $x$ we get $$ka_0y^nx^{k-1}=\beta f^ry^{2r}x^{n-2r}$$ for some $r$, where $\beta\in k[\Delta, e]$ implying that $f^r$ is a maximal vector.
This is false unless $r=0$, thus $k=0$. Writing $B$ as a sum of elements of type $\alpha_it^i[x,\Delta]x^{n-2i-1},$ we see by the preceding remark that it must contain a nonzero term of the form $\alpha t^{n-1/2}[x,\Delta]$ with $\alpha\in R[\Delta].$ In particular, $n$ must be odd, hence $A$ may not contain any terms with no powers of $x$ in it. If we can show that $[B, f]$ has such a term, we will have a contradiction. We have $\alpha t^{n-1/2}[[\Delta, x], f]=t^{n-1/2}[\alpha[\Delta, x],f],$ since weight of $B$=0, $\alpha=\gamma e^m$, where $m=\frac{1}{2}(p-1), \gamma\in R[\Delta].$ Hence it would suffice to show that $[e^k[\Delta, x], f]$ contains $y$. We have $$[e^k[\Delta,x], f]=-e^k[\Delta, y]+[e^k, f][\Delta, x].$$ Grouping $y$ terms in the above we get the following term 
$$([e^k, f]4e-e^k(2h+3))y.$$ Straightforward computation shows that this term is not 0.

\end{proof}

\begin{lemma} If an element $A\in F^{p-1}H_z$ commutes with $\mf{g},$ then it belongs to $R[\Delta,t_z]$.

\end{lemma}
\begin{proof}We may pass to the associated graded, so once again we may assume that $z=0$. As explained before $A$ may not be divisible by either $x, y$. Assuming that it is homogeneous of degree $n$, we may write
$$A=\sum \alpha_ix^{n-i}y^i, \alpha_0,\\ \alpha_n\neq 0, \alpha_0=\alpha e^m, \alpha_n=\beta e^m$$ where $\alpha, \beta\in R[\Delta]$ and $n=2m$, Hence
$A-\alpha t^m$ is divisible by $x,$ so $A=\alpha t^m$.

\end{proof}

\begin{lemma} Let $A=\sum_i\alpha_i t^i[\phi_i,x^{n-2i}], \alpha_i\in R[e]$ be a homogeneous element (in x, y) of weight 0 such that $\phi_i\in \Delta k[\Delta].$ Then $[A,\Delta]=0$ if and only if all $\alpha_i$ are 0.
\begin{proof}
Since $A$ commutes with $e$ and $h$, we get that it must lie in the centralizer of $\mf{g}$, hence by the previous lemma we have 
$$A=\sum_i\alpha_i t^i[\phi_i,x^{n-2i}]=\beta t^{\frac{n}{2}}, \beta\in R[\Delta].$$ This  equality implies in particular that $\alpha_0[\phi_0,x^n]$ is divisible by $t.$ Let us show that this cannot be the case. Consider the universal enveloping algebra filtration on $H$ (recall that $H=\mf{U}(\mathfrak{sl}_2\ltimes V)$). The top symbol of $\alpha_0[\phi_0, x^n]$ with respect to this filtration is equal to $$n(F(\phi_0)h+G(\phi_0)x+2eF(\phi_0)y)x^{n-1},$$ which is clearly not divisible by $t=ey^2+hxy-fx^2$ in $\Gr (H)=\Sym (\mathfrak{sl}_2\oplus V)$, hence $\alpha_0[\phi_0, x^n]$ is not divisible by $t$ in $H$.

\end{proof}

\end{lemma}
Next, we will show that elements $1, t, ..., t^{p-1}$ are linearly independent over the ring $\mf{Z}_0(H)=k[e^p, f^p, h^p-h, x^p, y^p]$ (which by itself is a polynomial ring in five variables). Indeed, let $0=\sum_{i<p} a_i t^i$, where $a_i\in \mf{Z}_0(H).$ We may assume that the elements $a_i$ are homogeneous in $x, y$ and the whole expression in the summation has the same degree $n$ (in $x, y$). Thus we have \\$n=$deg$(a_i)+2i$, and since deg($a_i$) are multiples of $p$, we have $2i=n$mod$(p$). We deduce that at most one $a_i\neq 0$, hence all $a_i=0$.

 Now we claim that $$ t^p=e^p(x^p)^2-f^p(y^p)^2+(h^p-h)x^py^p.$$ Indeed, let us consider the following central element $a=t^p-
(e^p(x^p)^2-f^p(x^p)^2).$ It is clear that this element is divisible by both $x,y$ from the right, therefore from the argument above
we see that it must be divisible by $x^py^p.$ Since $a$ has degree $2p,$ we must have $$t^p= e^p(x^p)^2-f^p(x^p)^2 +\omega x^p y^p,$$ where 
$\omega$ is some element in $k[e^p, h^p-h, f^p].$ Now let us consider the enveloping algebra filtration on $H$, using the well-known fact that \\$(a+b)^p=a^p+b^p$mod$A/[A,A]$ for any algebra $A$ over $k$ and any elements $a, b\in A$, we see that $t^p-(e^p(x^p)^2-f^p(x^p)^2-(h^p-h)x^py^p)$ has the filtration degree less than $3p.$ Hence, we see that $$t^p-(e^p(x^p)^2-f^p(x^p)^2+(h^p-h)x^py^p)=\alpha x^py^p$$ where $\alpha$ is some constant. We claim that this constant must be 0. Indeed, if we write $t^p$ as the sum of PBW monomials ($x, y$ on the right), then it becomes clear that each monomial term appearing in $t^p$ belongs to $\mf{g}H.$ Thus we may conclude that $\alpha=0$. To conclude, we have shown that $\mf{Z}(H)$ is isomorphic to $k[x_1,x_2,x_3,x_4,x_5,y]/(y^p-x_1x_2^2-x_3x_4^2+x_5x_2x_4)$, thus $\spec (\mf{Z}(H))$ is a (ramified) degree $p$ covering of $\mathbb{A}^5$.

Now let us consider the case of an arbitrary  $z.$ We will show that elements $1, t_z, ..., t_z^{p-1}$ are linearly independent over $\mf{Z}_0(H_z)=k[e^p, f^p, h^p, x_p, y_p]$ and generate $\mf{Z}(H_z)$ as a module over $\mf{Z}_0(H_z)$. This immediately follows from the corresponding statement for $1, t, ...t^{p-1}$ over $\Gr (\mf{Z}(H_z))$. Thus there exists a monic polynomial $f(\tau)$ in one variable $\tau$ of degree $p$ such that $\mf{Z}(H_z)$ is isomorphic to $\mf{Z}_0(H_z)[\tau]/(f(\tau)).$ Thus Spec($\mf{Z}(H_z)$) is a degree $p$ covering of $\spec \mf{Z}_0(H_z)$)=$\mathbb{A}^5$.

 Let us write an explicit formula for the central element $x_p$ when the parameter $z$ is at most linear. Clearly, if $z$ is a constant, then already $x^p$ is central, thus we assume that $z=\Delta + a, a\in k$. We have $$[z,x]=(2h-3)x+4ey,$$ so $$[[z,x],x]=(\ad x)^2(z)=2x^2-4ez,$$ hence $$(\ad x)^{2n}(z)=(-4e)^{n-1}(2x^2-4ez)$$ and 
\begin{eqnarray*}
(\ad x)^{2n+1}(z)=[x,(\ad x)^{2n}(z)]&=&(-4e)^{n-1}(4e(2h-3)x+4ey))\\&=&(-1)^{n-1}(4e)^n((2h-3)x+4ey). 
\end{eqnarray*}
Let us write $p$ as $2k+1.$ We will make the following computation: $$[f,e^kx]=e^ky-khe^{k-1}x+k(k-1)e^{k-1}x.$$ We claim that $[f,x^p+(-4)^ke^kx]=0.$ Indeed $$[f,x^p]=-[x^p,f]=-(\ad x)^p(f)=(\ad x)^{p-2}(z).$$ Thus $$[f,x^p]=(-1)^k(4e)^{k-1}((2h-3)x+4ey),$$ and 
\begin{eqnarray*}
[f,-(-4)^ke^kx]=-4^ke^ky+4^kkhe^{k-1}x^(k-1)ke^{k-1}x=[f, x^p]. 
\end{eqnarray*}
As a result, we get that $x_p=x^p+(-4)^ke^kx$ lies in the center of $H_z$.
\begin{cor} Algebra $H_z$ is a prime Noetherian ring which is Auslander regular and Cohen-Macaulay, and it is finite dimensional over its center and its PI-degree is equal to $p^2$

\end{cor}
\begin{proof}
After passing to the associated graded and applying results of [BG], the only thing which needs to be checked is the part about the center. It is well-known that $H$ being the enveloping algebra of a restricted Lie algebra is a free module of dimension $p^5$ over $\mf{Z}_0(H).$ Hence again by the associated graded argument, $H_z$ is a free module of dimension $p^5$ over $\mf{Z}_0.$ Since $[K(\mf{Z}(H_z)):K(\mf{Z}_0)]=p$ where $K(-)$ denotes the field of fractions, we see that the PI-degree of $H_z$ is equal to $p^2$. 
 \end{proof} 
Now we turn to the study of irreducible representations of $H_z.$ The main result is the following

\begin{theorem}For any $z\in k[\Delta]$ such that deg$(z)<p-1,$ the Azumaya and the smooth loci of $H_z$ coincide.

\end{theorem}
\begin{proof}
By the previous result, $H_z$ has PI-degree equal to $p^2.$ Hence by the results of [BG], it suffices to show
that there exists an open set $U\subset Spec(\mf{Z}(H_z))$ whose complement has codimension at least 2 in $\mf{Z}(H_z),$ such that
for any $\chi\in U,$ algebra $H_z$ has an irreducible representation of dimension $\geq p^2.$

Let us consider a set of characters $\chi\in Spec\mf{Z}(H_z)$ for which $\chi(e^p)=\chi(f^p)=0, \chi(h^p-h)\neq 0$ thus $\mf{sl}_2$ part is regular semisimple, and $\chi(x_p)\neq 0$ or $\chi(y_p)\neq 0$. We claim that any such $\chi$ is in the Azumaya locus. Indeed, we may assume without loss of generality that $\chi(x_p)\neq 0.$ Let
$V_{\chi}$ be an irreducible module for $\chi.$ Since $h$ acts diagonalizably and $e$ acts nilpotently on $V_{\chi},$ there exists a nonzero element $v\in V_{\chi}$ such that $ev=0, hv=\lambda v$ for some $\lambda\in k$
(evidently $\lambda^p-\lambda=\chi(h^p-h)$.) Then, clearly $ex^iv=0, h(x^iv)=(i+\lambda)x^iv$ for any $i=0,..., p-1.$ It follows from the description of $x_p$ that $x^pv=x_pv=\chi(x_p)v\neq 0,$ in particular all $x_i$ are nonzero.
Therefore $V_{\chi}$ considered as a module over $\mf{sl}_2$ has $V(\lambda+i)$ as a submodule, for all $i=0, ..., p-1$  
(where $V(\lambda+i)$ is an irreducible module over $\mf{sl}_2$ for the regular
semi-simple character $\chi$ corresponding to weight $\lambda+i$ [FP]).  Thus, dim$V_{\chi}\geq p^2,$ hence
$\chi$ is in the Azumaya locus. 

 Next let us consider characters $\chi$ with the property $\chi(e^p)=\chi(x_p)=\chi(h^p-h)=0.$ 
 Let us denote by $H_{\chi}$ the fiber of $H_z$ at the point $\chi\in Spec(\mf{Z}(H_z)).$
 Let us consider the
 module $M_{\chi}=H_{\chi}\otimes_{B_{\chi}}k_{\chi}$ (analog of the baby Verma module),
 where $k_{\chi}=kv, xv=ev=0, hv=0,$ and $B_{\chi}$ is a subalgebra of $H_{\chi}$ generated by $e, x, h.$
 It is clear that dim$M_{\chi}=p^2$ and $M_{\chi}$ is spanned by elements
 $f^iy^jv, i, j<p.$ Our goal will be to produce $\beta^p=\chi(y_p)$ such that $M_{\chi}$ is irreducible. Let
 us assume that it is not irreducible. Let us choose a homogeneous weight element (with respect to ad$(h)$) $g\in k[f, y]$ of smallest
 total degree in $f, y$ such that $H_zgv$ is a proper submodule of $M_{\chi}.$ It is clear
 that we may choose such $g$ with the property that $egv=xgv=0.$
It is also clear that $y$ is invertible on $M_{\chi}.$ Therefore we may write 
$$g=\sum_{i=0}^{n} a_if^iy^{2n-2i}$$ 
where $a_0=1, a_n\neq 0, a_i\in k, 2n<p.$ It is easy to see that such $a_i$ are determined uniquely (follows from 
$[x, g]v=0$), in particular, they do not depend on $\beta$. Let us write $y_p$ as $y_p=y^p+fy^{p-2}b_0+...+f^kyb_k,$ where $k=\frac{1}{2}(p-1)$ 
and $b_i\in k[\Delta].$ Then we have the equality $y^p-c_0fy^{p-2}-...-\beta^p=0$ in the endomorphism ring of $M_{\chi}$.
Let $L_{\chi}=M_{\chi}/H_zgv,$ then $L(\chi)=k[f, y]/I$ where $I=Ann(L_{\chi})\cap k[f, y].$ We have that
$f^p-1, g\in I$, clearly we may choose $\beta$ so that $y^p+c_0fy^{p-2}+...\beta^p$ is invertible in $k[f, y]/I$
thus $L_{\chi}=0,$ a contradiction. Thus $M_{\chi}$ is irreducible for generic values of $\beta$.

 Let us summarize: Let $U$ be the $SL_2$-orbit (under the adjoint action) of characters $\chi$ such that
 $\chi(e^p)=\chi(f^p)=0, \chi(h^p-h)\neq 0$ and either $\chi(x_p)$ or $\chi(y_p)$ is nonzero. Also, let
 $V$ be an $SL_2$-orbit of characters $\chi$ such that $\chi(f^p)=1, \chi(h^p-h)=0, \chi(e^p)=0.$ Then we proved that $U$ is in the Azumaya locus, and the Azumaya locus has a nonempty intersection with $V.$ Now since
 $V$ has codimension 1, and the complement of $U\cup V$ has codimension 2 (in Spec($\mf{Z}(H_z)$)), we may deduce
 that the complement of the Azumaya locus has codimension $\geq 2.$ Hence, by the results of [BG], we are done.
\end{proof} 
In what follows we explicitly describe all irreducible modules for the case $z=0.$ Thus for a given central
character $\chi,$ we want to describe corresponding irreducible modules of $H_{\chi}.$ We may assume
without loss of generality that  a central character $\chi$  satisfies the equality $\chi(e^p)=0$ (using $SL_2(k)$-action.) We will adopt the following
notation. By $N_{\chi}$ we denote a subalgebra of $H_{\chi}$ generated by $f, y$ and $k_{\chi}$ will denote
its one dimensional representation $kv=k_{\chi}$ such that $f^pv=\chi(f^p)v, y^pv=\chi(y^p)v.$

\begin{prop} Let $\chi\in Spec(\mf{Z}(H_z))$ be a character such that $\chi(x^p)\neq 0,$ then the module
$M_{\chi}=H_{\chi}\otimes_{N_{\chi}}k_{\chi}$ is the irreducible rank $p^2$ module over $H_{\chi}.$ If $\chi(x^p)=0$ and $\chi(y^p)\neq 0$ then
$M_{\chi}=H_{\chi}\otimes_{B_{\chi}}k_{\chi}$ is the irreducible module of rank $p^2.$ If $\chi(x^p)=\chi(y^p)=0$ then any irreducible
$H_{\chi}$ modules is an irreducible $\mf{sl}_2$-module on which $x, y$ act as 0.
\end{prop}
\begin{proof}

From the description of the center of $H$ it follows that the singular locus of $\mf{Z}(H)$ is $x^p=y^p=0.$ 
If $\chi\in Spec(\mf{Z}(H))$ belongs to the singular locus, and $V_{\chi}$ is an irreducible module of $H_{\chi},$
then it is clear that there exists a nonzero $v\in V_{\chi},$ such that $xv=yv=0,$ which implies that $VV_{\chi}=0$
(recall that $V=kx\oplus ky$).
Thus $H_{\chi}$ is an irreducible $\mf{sl}_2$-module.

 Let a character $\chi$ be in the smooth locus, thus by the previous theorem, it belongs to the Azumaya
 locus. Hence we only need to show that modules in the proposition are nonzero and have dimension $\leq p^2.$

Let us start with the case when $\chi(x^p)=0.$ It is clear that dim $(M_{\chi})=p^2$, so there is nothing to prove.
Next, we consider the case when  $\chi(y^p)=0.$ 
 Thus $M_{\chi}$ is spanned by elements of the form $e^ix^jh^lv.$ But we claim that $M_{\chi}=k[e, x]v.$ Indeed, since $t=ey^2+hxy-fx^2$ acts as a scalar, it follows that multiplication by $hxy$ preserves $k[e, x]v, $ so if $\chi(y^p)\neq 0$  then $hv\in k[e, x]v.$ If $\chi(y^p)=0,$ then we have $hxyex^{p-2}v=-hx^{p}v.$
Thus, $hv\in k[e,x]v,$ so $M_{\chi}=k[e, x]v,$ in particular dim$M_{\chi}\leq p^2.$ Now let us assume that 
$\chi(y^p)\neq 0.$ Then from the fact that $hxyv\in k[e, y]v$ we see that $hv\in k[e, y],$ hence 
$M_{\chi}=k[e, x]v.$
It is clear that an irreducible module over $H_{\chi}$ is a quotient of $M_{\chi},$ thus $V_{\chi}=M_{\chi}$ and we are done.

\end{proof}

\end{document}